\RequirePackage{fix-cm}
\documentclass[smallextended]{svjour3}       
\smartqed  
\usepackage{graphicx}

\usepackage{amsmath,amssymb}
\usepackage{url}

\usepackage{algorithm}
\usepackage{algpseudocode}
\usepackage{color}
\usepackage{tikz}
\usepackage{subfigure}
\usetikzlibrary{arrows,decorations.pathmorphing,backgrounds,fit,positioning,shapes.symbols,chains}
\definecolor{darkred}{rgb}{0.7,0,0}
\definecolor{lightred}{rgb}{0.5,0,0}
\definecolor{lightgreen}{rgb}{0,0.8,0}
\definecolor{mathset}{HTML}{CFD9E8}
\definecolor{mathboundary}{HTML}{5E81B5}
\definecolor{remove}{HTML}{E34A33}
\definecolor{edit}{HTML}{0000AA}

\usepackage[normalem]{ulem}

\tikzstyle{frac}=[rectangle,draw=black,fill=lightgreen,
inner sep=0pt,minimum size=0.1cm]
\tikzstyle{active}=[draw=darkred,very thick]
\tikzstyle{inactive}=[draw=lightred, dotted,very thick]
\tikzstyle{notactive}=[draw=mathboundary,thick, fill=mathset]
\tikzstyle{notactiveinner}=[draw=mathboundary,thick,fill=white]
\begin{document}

\title{Polyhedral approximation in mixed-integer convex optimization
}

\author{Miles Lubin \and Emre Yamangil \and \\ Russell Bent \and Juan Pablo Vielma
}


\institute{M. Lubin \and J.\ P. Vielma \at
              Massachusetts Institute of Technology, Cambridge, MA, USA \\
              \email{mlubin@mit.edu}           
           \and
           E. Yamangil \and R. Bent \at
              Los Alamos National Laboratory, Los Alamos, NM, USA
}

\date{Received: date / Accepted: date}

\maketitle

\begin{abstract}
    Generalizing both mixed-integer linear optimization and convex optimization, \textit{mixed-integer convex optimization} possesses broad modeling power but has seen relatively few advances in general-purpose solvers in recent years. In this paper, we intend to provide a broadly accessible introduction to our recent work in developing algorithms and software for this problem class. Our approach is based on constructing polyhedral outer approximations of the convex constraints, resulting in a global solution by solving a finite number of mixed-integer linear and continuous convex subproblems. The key advance we present is to strengthen the polyhedral approximations by constructing them in a higher-dimensional space. In order to automate this \textit{extended formulation} we rely on the algebraic modeling technique of disciplined convex programming (DCP), and for generality and ease of implementation we use conic representations of the convex constraints. Although our framework requires a manual translation of existing models into DCP form, after performing this transformation on the MINLPLIB2 benchmark library we were able to solve a number of unsolved instances and on many other instances achieve superior performance compared with state-of-the-art solvers like Bonmin, SCIP, and Artelys Knitro.
\keywords{Convex MINLP \and Outer approximation \and Disciplined convex programming}
\end{abstract}

\section{Introduction}

Mixed-integer linear programming (MILP) has established itself as a practical framework for optimization problems in scheduling, logistics, planning, and many other areas. Although these problems are in general NP-Hard, more than 50 years of investment in MILP techniques has resulted in powerful commercial and open-source solvers that can solve MILP problems of practical interest within reasonable time limits~\cite{50years}. The aim of this paper is to develop methodologies for solving the more general class of \textit{mixed-integer convex optimization}---or mixed-integer convex programming (MICP)---problems by reducing them to a sequence of MILP problems.

In order to employ MILP, we relax the convex constraints by representing them as an intersection of a finite number of half-spaces, that is, polyhedral constraints.
Based on this idea, Duran and Grossman~\cite{DuranGrossmann} and Leyffer~\cite{SvenThesis} developed the \textit{outer approximation} (OA) algorithm which solves a sequence of MILP and continuous, convex subproblems to deliver a globally optimal solution for MICP problems in a finite number of iterations; we present a generalized version of this algorithm in Section~\ref{sec:OA}.

Despite the fact that many MICP approaches, including the OA algorithm, build on MILP approaches, there remains a significant performance gap between the two problem classes. Bonami, Kilin\c{c}, and Linderoth~\cite{BonamiReview} note in a recent review that continued advances in MILP have translated into ``far more modest'' growth in the scale of problems which MICP solvers can solve within reasonable time limits. Hence, despite numerous potential applications (see the reviews~\cite{BonamiReview,MahajanReview}), our perception is that MICP has not entered the mainstream of optimization techniques, perhaps with the exception of the special case of \textit{mixed-integer second-order cone programming} (MISOCP) which we will discuss at length.

The cases in which the OA algorithm and others based on polyhedral approximation perform poorly are those in which the convex set of constraints is poorly approximated by a small number of half-spaces. In Section~\ref{sec:hijazi}, we review a simple example identified by Hijazi et al.~\cite{Hijazi} where the OA algorithm requires $2^n$ iterations to solve an MICP instance with $n$ decision variables. Fortunately, \cite{Hijazi} also propose a solution based on ideas from \cite{Baron} that can significantly improve the quality of a polyhedral approximation by constructing the approximation in a higher dimensional space. These constructions are known as \textit{extended formulations}, which have also been considered by~\cite{VielmaExtendedFormulations,MustafaThesis}. Although Hijazi et al. demonstrate impressive computational gains by using extended formulations, implementing these techniques within traditional MICP solvers requires more structural information than provided by ``black-box'' oracles through which these solvers typically interact with nonlinear functions. To our knowledge, MINOTAUR~\cite{minotaur} is the only such solver to automate extended formulations. In Section~\ref{sec:dcp} we identify the modeling concept of disciplined convex programming (DCP)~\cite{DCP}, popularized in the CVX software package~\cite{cvx}, as a practical solution to the problem of automatically generating extended formulations based on a user's algebraic representation of an MICP problem.

Our investigation of DCP leads us in Section~\ref{sec:micone} to consider \textit{conic} optimization as a representation of convex constraints that compactly encodes all of the information needed to construct extended formulations. This key observation links together a number of streams in convex optimization and MICP research, and in particular explains the increasingly popular role of MISOCP and how it can be extended to cover ``general'' MICP.  Pulling these pieces together, in Section~\ref{sec:conicoa} we develop the first OA algorithm for \textit{mixed-integer conic} optimization problems based on conic duality. In Section 7, we present \textit{Pajarito}, a new solver for MICP based on the conic OA algorithm and compare its efficiency with state-of-the-art MICP solvers. We report the solution of a number of previously unsolved benchmark instances.

This paper is meant to be a self-contained introduction to all of the concepts beyond convex optimization and mixed-integer linear optimization needed to understand the algorithm implemented in Pajarito. Following broad interest in our initial work~\cite{IPCO}, we believe that a primary contribution of this paper is to compile the state of the art for readers and to tell a more detailed story of why DCP and conic representations are a natural fit for MICP. For example, in Section~\ref{sec:OA} we present the OA algorithm in a straightforward yet generic fashion not considered by previous authors that encompasses both the traditional smooth setting and the conic setting. A notable theoretical contribution beyond~\cite{IPCO} is an example in Section~\ref{sec:conicoa} of what may happen when the assumptions of the OA algorithm fail: an MICP instance for which \textit{no} polyhedral outer approximation is sufficient. Our computational results in Section~\ref{sec:computation} have been revised with more comparisons to existing state-of-the-art solvers, and as a final contribution above~\cite{IPCO}, our solver Pajarito has now been publicly released along with the data and scripts required to reproduce our experiments.

\section{State of the art: polyhedral outer approximation}\label{sec:OA}

We state a generic mixed-integer convex optimization problem as
\begin{align}\label{eq:miconv}
\min_{x} \quad& c^Tx  \notag \\
\text{s.t.} \quad& x \in X,\tag{MICONV}\\
&x_i \in \mathbb{Z}, l_i \le x_i \le u_i\quad \forall i \in I, \notag
\end{align}
\noindent
where $X$ is a closed, convex set, and the set $I \subseteq \{ 1, 2, \ldots, n \}$ indexes the integer-constrained variables, over which we have explicit finite bounds $l_i$ and $u_i$ for $i \in I$. We assume that the objective function is linear. This assumption is without loss of generality because, given a convex, nonlinear objective function $f(x)$, we may introduce an additional variable $t$, constrain $(t,x)$ to fall in the set $\{ (t,x) : f(x) \le t \}$, known as the \textit{epigraph} of $f$, and then take $t$ as the linear objective to minimize~\cite{BonamiReview}. For concreteness, the convex set of constraints $X$ could be specified as
\begin{equation}\label{eq:Xfunc}
X = \{ x \in \mathbb{R}^n : g_j(x) \le 0, j \in J \},
\end{equation}
for some set $J$ where each $g_j$ is a smooth, convex function, although we do not make this assumption. We refer to the constraints $x_i \in \mathbb{Z}\,\, \forall \, i \in I$ as \textit{integrality constraints}. Note that when these integrality constraints are relaxed (i.e., removed),~\ref{eq:miconv} becomes a convex optimization problem.

A straightforward approach for finding the global solution of~\eqref{eq:miconv} is branch and bound. Branch and bound is an enumerative algorithm where lower bounds derived from relaxing the integrality constraints in~\eqref{eq:miconv} are combined with recursively partitioning the space of possible integer solutions. The recursive partition is based on ``branches'' such as $x_i \le k$ and $x_i \ge k + 1$ for some integer-constrained index $i \in I$ and some value $k$ chosen between the lower bound $l_i$ and the upper bound $u_i$ of $x_i$. In the worst case, branch and bound requires enumerating all possible assignments of the integer variables, but in practice it can perform much better by effectively pruning the search tree. Gupta and Ravindran~\cite{Gupta85} describe an early implementation of branch-and-bound for MICP, and Bonami et al.~\cite{BonamiNLPBNB} more recently revisit this approach.

On many problems, however, the branch-and-bound algorithm is not competitive with an alternative family of approaches based on \textit{polyhedral outer approximation}. Driven by the availability of effective solvers for linear programming (LP) and MILP, it was observed in the early 1990s by Leyffer and others~\cite{SvenThesis} that it is often more effective to avoid solving convex, nonlinear relaxations, when possible, in favor of solving polyhedral relaxations using MILP. This idea has influenced a majority of the solvers recently reviewed and benchmarked by Bonami et al.~\cite{BonamiReview}.

In this section, we will provide a sketch of an \textit{outer approximation} (OA) algorithm. We derive the algorithm in a more general way than most authors that will later be useful in the discussion of mixed-integer conic problems in Section~\ref{sec:conicoa}, although for intuition and concreteness of the discussion we illustrate the key points of the algorithm for the case of the smooth, convex representation~\eqref{eq:Xfunc}, which is the traditional setting. We refer readers to~\cite{Bonmin,DuranGrossmann,Filmint} for a more rigorous treatment of the traditional setting and Section~\ref{sec:conicoa} for more on the conic setting (i.e., when $X$ is an intersection of convex cone and an affine subspace). We begin by defining polyhedral outer approximations.

\begin{definition}
A set $P$ is a polyhedral outer approximation of a convex set $X$ if $P$ is a polyhedron (an intersection of a finite number of closed half-spaces, i.e., linear inequalities of the form $a_i^Tx \le b_i$) and $P$ contains $X$, i.e., $X \subseteq P$.
\end{definition}

Note that we have not specified the explicit form of the polyhedron. While the traditional OA algorithm imagines $P$ to be of the form $\{ x \in \mathbb{R} : Ax \le b\}$ for some $A$ and $b$, it is useful to not tie ourselves, for now, to a particular representation of the polyhedra.

Polyhedral outer approximations of convex sets are quite natural in the sense that every closed convex set can be represented as an intersection of an \textit{infinite} number of closed half-spaces~\cite{hiriart-lemarechal-1996}. For instance, when $X$ is given in the functional form~\eqref{eq:Xfunc} and each $g_j : \mathbb{R}^n \to \mathbb{R}$ is smooth and finite-valued over $\mathbb{R}^n$ then the following equivalence holds:
\begin{equation}\label{eq:finiteoa}
X = \{ x \in \mathbb{R}^n : g_j(x') + \nabla g_j(x')^T(x-x') \le 0\,\, \forall \, x' \in \mathbb{R}^n, j \in J \},
\end{equation}
where $\nabla g_j(x')$ is the gradient of $g_j$. When some $g_j$ functions are not defined (or do not take finite values) over all of $\mathbb{R}^n$ then these ``gradient inequalities'' plus additional linear constraints enforcing the domain of each $g_j$ provide a representation of $X$ as an intersection of halfspaces; see~\cite{hiriart-lemarechal-1996} for further discussion.

Hence, in the most basic case, a polyhedral approximation of $X$ can be derived by picking a finite number of points $S \subset \mathbb{R}^n$ and collecting  the half-spaces in~\eqref{eq:finiteoa} for $x' \in S$ instead of for all $x' \in \mathbb{R}^n$. What is perhaps surprising is that a finite number of half-spaces provides a sufficient representation of $X$ in order to solve~\eqref{eq:miconv} to global optimality, under some assumptions\footnote{In Lemma~\ref{claim:counterexample} we provide an explicit counterexample.}. This idea is encompassed by the OA algorithm which we now describe. 

Given a polyhedral outer approximation $P$ of the constraint set $X$, we define the following mixed-integer linear \textit{relaxation} of~\eqref{eq:miconv}

\begin{align}\label{eq:mioa1}
r_P = \min_{x}\quad & c^Tx  \notag \\
\text{s.t.} \quad& x \in P,\tag{MIOA(P)}\\
&x_i \in \mathbb{Z},\, l_i \le x_i \le u_i\quad \forall i \in I. \notag
\end{align}

Note that~\ref{eq:mioa1} is a relaxation because any $x$ feasible to~\eqref{eq:miconv} must be feasible to~\ref{eq:mioa1}. Therefore the optimal value of~\ref{eq:mioa1} provides a lower bound on the optimal value of~\eqref{eq:miconv}. This bound \textit{may} be NP-Hard to compute, since it requires solving a mixed-integer linear optimization problem; nevertheless we may use existing, powerful MILP solvers for these relaxations.

For notational convenience, we sometimes split the integer-constrained components and the continuous components of $x$, respectively, writing $x = (x_I,x_{\bar I})$ where $\bar I = \{1, \ldots, n\} \setminus I$.
Given a solution $x^* = (x_I^*,x_{\bar I}^*)$ to~\ref{eq:mioa1}, the OA algorithm proceeds to solve the continuous, convex problem~\ref{eq:convsub} that results from fixing the integer-constrained variables $x_I$ to their values in $x^*_I$:
\begin{align}\label{eq:convsub}
v_{x_I^*} = \min\quad & c^Tx  \notag \\
\text{s.t.} \quad& x \in X,\tag{CONV($x^*_I$)}\\
&x_I = x_I^*. \notag
\end{align}

If~\ref{eq:convsub} is feasible, let $x'$ be the optimal solution. Then $x'$ is a feasible solution to~\eqref{eq:miconv} and provides a corresponding upper bound on the best possible objective value. If the objective value of this convex subproblem equals the objective value of~\ref{eq:mioa1} (i.e., $c^Tx' = c^Tx^*$), then $x'$ is a globally optimal solution of~\eqref{eq:miconv}. If there is a gap, then the OA algorithm must update the polyhedral outer approximation $P$ and re-solve~\ref{eq:mioa1} with a tighter approximation, yielding a nondecreasing sequence of lower bounds.

To ensure finite termination of OA it is sufficient to prevent
repetition of unique assignments of the integer-valued components $x_I^*$, because there is only a finite number of possible values. The following lemma states a condition on the polyhedral outer approximation $P$ that helps prove finite convergence.

\begin{lemma}\label{lem:oafinite}
Fixing $x_I \in \mathbb{Z}^{|I|}$, if $x = (x_I,x_{\bar I}) \in P$ implies $c^Tx \ge v_{x_I}$ for all $x_{\bar I} \in \mathbb{R}^{n-|I|}$ where $v_{x_I}$ is the optimal value of (CONV($x_I$)) then the OA algorithm must terminate if~\ref{eq:mioa1} returns an optimal solution $x^*$ with integer components matching $x^*_I = x_I$.
\end{lemma}
\begin{proof}
Assume we solve~\ref{eq:mioa1} and obtain a solution $x^*$. If the integer part of $x^*$ matches $x_I$, by our assumptions we have $r_P = c^Tx^* \ge v_{x_I}$, where $r_P$ is the optimal value of~\ref{eq:mioa1}. Since~\ref{eq:mioa1} is a relaxation and $v_{x_I}$ is the objective value of a feasible solution, then we must have $r_P = v_{x_I}$. Thus, we have proven global optimality of this feasible solution and terminate.
\end{proof}

Note that Lemma~\ref{lem:oafinite} provides a general condition that does not assume any particular representation of the convex constraints $X$. In the traditional setting of the smooth, convex representation~\eqref{eq:Xfunc}, if $x'$ is an optimal solution to~\ref{eq:convsub} \textit{and strong duality holds}, e.g., as in Prop 5.1.5 of Bertsekas~\cite{bertsekas}, then the set of constraints
\begin{equation}\label{eq:kkt}
g_j(x') + \nabla g_j(x')^T(x-x') \le 0\,\, \forall \, j \in J
\end{equation}
are sufficient to enforce the condition in Lemma~\ref{lem:oafinite} for finite convergence. In other words, within the OA loop after solving~\ref{eq:convsub}, updating $P$ by adding the constraints~\eqref{eq:kkt} is sufficient to ensure that the integer solution $x_I^*$ does not repeat, except possibly at termination. 
Intuitively, strong duality in~\ref{eq:convsub} implies that there are no descent directions (over the continuous variables) from $x'$ which are feasible to a first-order approximation of the constraints $g_j(x) \le 0$ for $j \in J$~\cite{bertsekas}. Hence a point $x = (x_I,x_{\bar I})$ sharing the integer components $x_I = x_I^*$ must satisfy $c^T(x-x') \ge 0$ or precisely $c^Tx \ge v_{x_I^*}$. See~\cite{SvenThesis,DuranGrossmann,Filmint} for further discussion.

If~\ref{eq:convsub} is infeasible, then to ensure finite convergence it is important to refine the polyhedral approximation $P$ to exclude the corresponding integer point. That is, we update $P$ so that
\begin{equation}\label{eq:infeascase}
\{ x \in \mathbb{R}^n : x_I = x_I^* \} \cap P = \emptyset.
\end{equation}
In the traditional smooth setting, it is possible in the infeasible case to derive a set of constraints analogous to~\eqref{eq:kkt}, e.g., by solving an auxiliary feasibility problem where we also assume strong duality holds~\cite{Bonmin,Filmint}.  

To review,
the OA algorithm proceeds in a loop between the MILP relaxation~\ref{eq:mioa1} and the continuous subproblem with integer values fixed\\~\ref{eq:convsub}. The MILP relaxation provides lower bounds and feeds integer assignments to the continuous subproblem. The continuous subproblem yields feasible solutions \textit{and} sufficient information to update the polyhedral approximation in order to avoid repeating the same assignment of integer values. The algorithm is stated more formally in Algorithm~\ref{alg:oa} and illustrated in Figure~\ref{fig:oa}.

\begin{algorithm}[ht]\small
\caption{The polyhedral outer approximation (OA) algorithm}\label{alg:oa}
\begin{algorithmic}
\State \textbf{Initialize:} $z_U \leftarrow \infty, z_L \leftarrow -\infty$, polyhedron $P \supset X$ such that~\ref{eq:mioa1} is bounded. Fix convergence tolerance $\epsilon$.
\While{$z_U - z_L \ge \epsilon$}
\State Solve \ref{eq:mioa1}.
\If{\ref{eq:mioa1} is infeasible}
\State \eqref{eq:miconv} is infeasible, so terminate.
\EndIf
\State Let $x^*$ be the optimal solution of \ref{eq:mioa1} with objective value $w_T$.
\State Update lower bound $z_L \leftarrow w_T$.
\State Solve~\ref{eq:convsub}.
\If{\ref{eq:convsub} is feasible}
\State Let $x'$ be an optimal solution of~\ref{eq:convsub} with objective value $v_{x_I^*}$.
\State Derive polyhedron $Q$ satisfying  $x = (x_I^*,x_{\bar I}) \in Q$ implies $c^Tx \ge v_{x_I^*}$ for\State all $x_{\bar I} \in \mathbb{R}^{n-|I|}$ by using strong duality (e.g.,~\eqref{eq:kkt}).
\If{$v_{x_I^*} < z_U$}
\State Update upper bound $z_U \leftarrow v_{x_I^*}$.
\State Record $x'$ as the best known solution.
\EndIf
\ElsIf{\ref{eq:convsub} is infeasible}
\State Derive polyhedron $Q$ satisfying $\{ x \in \mathbb{R}^n : x_I = x_I^* \} \cap Q = \emptyset$.
\EndIf
\State Update $P \leftarrow P \cap Q$.
\EndWhile
\end{algorithmic}
\end{algorithm}

The efficiency of the OA algorithm is derived from the speed of solving the~\ref{eq:mioa1} problem by using state-of-the-art MILP solvers. Indeed, in 2014 benchmarks by Hans Mittelman, the OA algorithm implemented within Bonmin using CPLEX as the MILP solver was found to be the fastest among MICP solvers~\cite{hans}. In spite of taking advantage of MILP solvers, the \textit{traditional} OA algorithm suffers from the fact that the gradient inequalities~\eqref{eq:kkt} may not be sufficiently strong to ensure fast convergence. In the following section, we identify when these conditions may occur and how to work around them within the framework of OA.

\begin{figure}[t]

\begin{tikzpicture}
\path[use as bounding box] (-3,-3) rectangle (3,3);

\draw[->,very thick] (2.2,1.5)--++(15:0.5) node[midway,above,sloped]{$c$};


\draw[notactive] (0,0) circle (2.5cm);

\draw[ultra thick] (-2,-1.5) -- (-2, 1.5);
\draw[ultra thick] (-1,-2.2912) -- (-1, 2.2912);
\draw[ultra thick] (0,-2.5) -- (0, 2.5);
\draw[ultra thick] (1,-2.2912) -- (1, 2.2912);
\draw[ultra thick] (2,-1.5) -- (2, 1.5);

\draw[ultra thick,<->,draw=mathboundary] (1.489161,2.813727) -- (1.811137,2.9) -- (2.8,-0.79048) -- (2.4780246,-0.876753);



\node at (2.41481,0.647048) [frac,label=right:$x'$] {};

\node at (2,2.195) [frac,label=right:$x^*$] {}; 
\end{tikzpicture}
\begin{tikzpicture}
\path[use as bounding box] (-3,-3) rectangle (3,3);



\draw[notactive] (0,0) circle (2.5cm);

\draw[ultra thick] (-2,-1.5) -- (-2, 1.5);
\draw[ultra thick] (-1,-2.2912) -- (-1, 2.2912);
\draw[ultra thick] (0,-2.5) -- (0, 2.5);
\draw[ultra thick] (1,-2.2912) -- (1, 2.2912);
\draw[ultra thick] (2,-1.5) -- (2, 1.5);

\draw[ultra thick,<->,draw=mathboundary] (1.489161,2.813727) -- (1.811137,2.9) -- (2.8,-0.79048) -- (2.4780246,-0.876753);

\draw[ultra thick,<->,draw=mathboundary] (0.90833, 2.4) -- (1.175,2.6) -- (2.9,0.3) -- (2.63333,0.1);




\node at (2,1.5) [frac,label=left:$x'$] {}; 

\end{tikzpicture}
\caption{An illustration of the outer approximation algorithm. Here, we minimize a linear objective $c$ over the ball $x_1^2 + x_2^2 \le 2.5$ with $x_1$ integer constrained. On the left, the point $x'$ is the solution of the continuous relaxation, and we initialize the polyhedral outer approximation with the tangent at $x'$. We then solve the~\ref{eq:mioa1} subproblem, which yields $x^*$. Fixing $x_1=2$, we optimize over the circle and update the polyhedral approximation with the tangent at $x'$ (on the right). In the next iteration of the OA algorithm (not shown), we will prove global optimality of $x'$. }\label{fig:oa}
\end{figure}
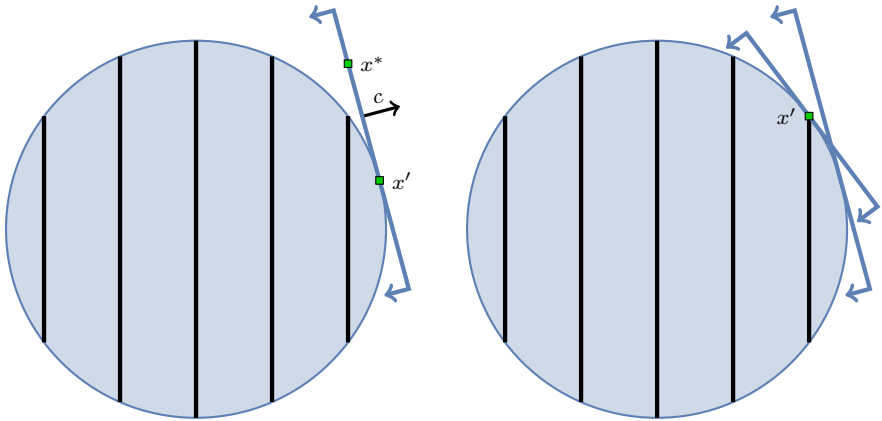

\section{State of the art: outer approximation enhancements}\label{sec:hijazi}

The outer approximation algorithm is powerful but relies on polyhedral outer approximations serving as good approximations of convex sets. The assumptions of the OA algorithm guarantee that there exists a polyhedron $P$ such that the optimal objective value of~\ref{eq:mioa1} matches the optimal objective value of~\eqref{eq:miconv}, precisely at convergence. In the case that~\eqref{eq:miconv} has no feasible solution, these assumptions furthermore guarantee that there exists an outer approximating polyhedron $P$ such that~\ref{eq:mioa1} has no feasible solution. In Section~\ref{sec:conicoa}, we discuss in more detail what may happen when the assumptions fail, although even in the typical case when they are satisfied, these polyhedra may have exponentially many constraints. Indeed, there are known cases where the OA algorithm requires $2^n$ iterations to converge in $\mathbb{R}^n$. In this section, we review an illustrative case where the OA algorithm performs poorly and the techniques from the literature that have been proposed to address this issue.

Figure~\ref{fig:hijazi} illustrates an example developed by Hijazi et al.~\cite{Hijazi}, specifically the problem,
\begin{align}\label{eq:hijazi}
\min_x\quad & c^Tx  \notag \\
\text{s.t.}\quad & \sum_{i=1}^n \left(x_i - \frac{1}{2}\right)^2 \le \frac{n-1}{4},\\
&x \in \mathbb{Z}^n, 0 \le x \le 1, \notag
\end{align}
which, regardless of the objective vector $c$, has no feasible solutions. Any polyhedral approximation of the single convex constraint, a simple ball, requires $2^n$ half-spaces until the corresponding outer approximation problem~\ref{eq:mioa1} has no feasible solution. At this point the OA algorithm terminates reporting infeasibility.

\begin{figure}[t]
\centering
\includegraphics[width=0.4\textwidth,clip=true,trim={1cm 2cm 2cm 1cm}]{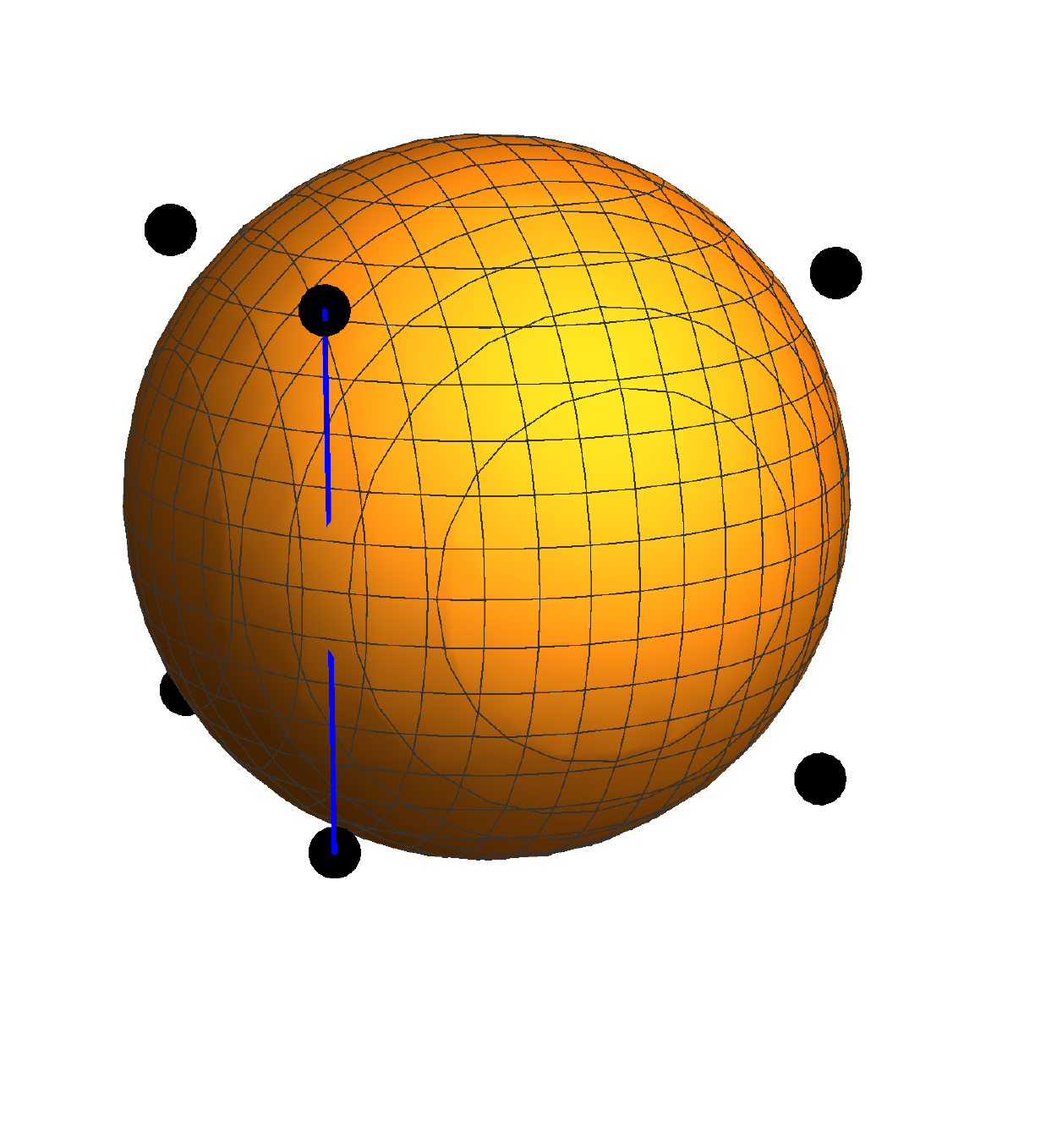}
\caption{The example developed by Hijazi et al.~\cite{Hijazi} demonstrating the case where the outer approximation algorithm requires $2^n$ iterations to converge in dimension $n$. The intersection of the ball with the integer lattice points (in black) is empty, yet any polyhedral outer approximation of the ball in $\mathbb{R}^n$ requires $2^n$ hyperplanes before it has an empty intersection with the integer lattice, because the line segments between any two lattice points (one of which is drawn) intersect the ball. Hence, any hyperplane can separate at most one lattice point from the ball, and we require $2^n$ of these to prove infeasibility. }\label{fig:hijazi}
\end{figure}

Hijazi et al. propose a simple yet powerful reformulation that addresses this poor convergence behavior. To motivate their reformulation, we recall a basic example from linear programming. The $\ell_1$ unit ball, i.e., $\{ x \in \mathbb{R}^n : \sum_{i=1}^n |x_i| \le 1 \}$, is representable as an intersection of half spaces in $\mathbb{R}^n$, namely the $2^n$ half spaces of the form $\sum_{i=1}^n s_ix_i \le 1$ where $s_i = \pm 1$. This exponentially large representation of the $\ell_1$ ball is seldom used in practice, however. Instead, it is common to introduce extra variables $z_i$ with constraints
\begin{equation}\label{eq:extabs}
z_i \ge x_i, z_i \ge -x_i \text{ for } i = 1, \ldots, n \text{ and } \sum_{i=1}^n z_i \le 1.
\end{equation}
It is not difficult to show that $||x||_1 \le 1$ if and only if there exist $z$ satisfying the constraints~\eqref{eq:extabs}. Note that these $2n+1$ constraints define a polyhedron in $\mathbb{R}^{2n}$, which we call an \textit{extended formulation} of the $\ell_1$ ball because the $\ell_1$ ball is precisely the projection of this polyhedron defined in $(x,z)$ space onto the space of $x$ variables. It is well known that polyhedra, such as the $\ell_1$ ball, which require a large description as half-spaces in $\mathbb{R}^n$ might require many fewer half-spaces to represent if additional variables are introduced~\cite{kippmartin}. Note, in this case, that the extended formulation is derived by introducing a variable $z_i$ to represent the epigraph $\{ (z,x) : |x| \le z \}$ of each $|x_i|$ term, taking advantage of the fact that the $\ell_1$ ball can be represented as a constraint on a sum of these univariate functions.

The solution proposed by Hijazi et al. and earlier by Tawarmalani and Sahinidis~\cite{Baron} follows this line of reasoning by introducing an extended formulation for the polyhedral representation of the smooth $\ell_2$ ball. Analogously to the case of the $\ell_1$ ball, Hijazi et al. construct an outer-approximating polyhedron in $\mathbb{R}^{2n}$ with $2n+1$ constraints which contains no integer points. By the previous discussion, we know that the projection of this small polyhedron in $\mathbb{R}^{2n}$ must have at least $2^n$ inequalities in $\mathbb{R}^n$. Their solution precisely exploits the separability structure in the definition of the $\ell_2$ ball, introducing an extra variable $z_i$ for each term and solving instead
\begin{align}\label{eq:hijaziext}
\min_{x,z}\quad & c^Tx  \notag \\
\text{s.t.}\quad & \sum_{i=1}^n z_i  \le \frac{n-1}{4},\\
& z_i \ge \left(x_i - \frac{1}{2}\right)^2, \quad \forall\, i=1,\ldots,n \notag\\
&x \in \mathbb{Z}^n, 0 \le x \le 1. \notag
\end{align}

The OA algorithm applied to~\eqref{eq:hijaziext} proves infeasibility in $2$ iterations because it constructs polyhedral approximations (based on gradient inequalities~\eqref{eq:kkt}) to the constraints in the $(x,z)$ space. More generally, Hijazi et al. and  Tawarmalani and Sahinidis propose to reformulate any convex constraint of the form $\sum_i f_i(x_i) \le k$ as $\sum_i z_i \le k$ and $z_i \ge f_i(x_i)$ for each $i$ where $f_i$ are univariate convex functions. Just by performing this simple transformation before providing the problem to the OA algorithm, they are able to achieve impressive computational gains in reducing the time to solution and number of iterations of the algorithm.

Building on the ideas of Hijazi et al. and Tawarmalani and Sahinidis, Vielma et al.~\cite{VielmaExtendedFormulations} propose an extended formulation for the \textit{second-order cone} $\{ (t,x) \in \mathbb{R}^{n+1} : ||x||_2 \le t \}$, which is not immediately representable as a sum of univariate convex functions. They recognize that the second-order cone is indeed representable as a sum of bivariate convex functions, i.e., $\sum_i \frac{x_i^2}{t} \le t$, after squaring both sides and dividing by $t$. They obtain an extended formulation by introducing auxiliary variables $z_i \ge \frac{x_i^2}{t}$ and constrain $\sum_i z_i \le t$. This simple transformation was subsequently implemented by commercial solvers for \mbox{MISOCP} like Gurobi~\cite{gurobidisagg}, CPLEX~\cite{cplexdisagg}, and Xpress~\cite{xpresssocp}, yielding significant improvements on their internal and public benchmarks.

In spite of the promising computational results of Hijazi et al.\ first reported in 2011 and the more recent extension by Vielma et al., to our knowledge, MINOTAUR~\cite{minotaur} is the only general MICP solver which has implemented these techniques in an automated way. To understand why others like Bonmin~\cite{Bonmin} have not done so, it is important to realize that MICP solvers historically have had no concept of the mathematical or algebraic structure behind their constraints, instead viewing them through black-box oracles to query first-order and possibly second-order derivative values. The summation structure we 
exploit, which is algebraic in nature, is simply not available when viewed through this form, making it quite difficult to retrofit this functionality into the existing architectures of MICP solvers. In the following section, we will propose a substantially different representation of mixed-integer convex optimization problems that is a natural fit for extended formulations.

\section{Disciplined Convex Programming (DCP) as a solution}\label{sec:dcp}

In order to implement the extended formulation proposal of~\cite{Hijazi} in an automated way, one may be led to attempt a direct analysis of a user's algebraic representation of the convex constraints in a problem. However, this approach is far from straightforward. First of all, the problem of \textit{convexity detection} is necessary as a subroutine, because it is only correct to exploit summation structure of a convex function $h(x) = f(x) + g(x)$ when \textit{both} $f$ and $g$ are convex. This is not a necessary condition for the convexity of $h$; consider $f(x) = x_1^2 - x_2^2$ and $g(x) = 2x_2^2$. Convexity detection of algebraic expressions is NP-Hard~\cite{ConvexityNPHard}, which poses challenges for implementing such an approach in a reliable and scalable way. \textit{Ad-hoc} approaches~\cite{DrAmpl} are possible but are highly sensitive to the form in which the user inputs the problem; for example, approaches based on composition rules fail to recognize convexity of $\sqrt{x_1^2 + x_2^2}$ and $\log(\exp(x_1)+\exp(x_2))$~\cite{Baron}.

Instead of attempting such analyses of arbitrary algebraic representations of convex functions, we propose to use the modeling concept of disciplined convex programming (DCP) first proposed by Grant, Boyd, and Ye~\cite{DCP,gb08}. In short, DCP solves the problem of convexity detection by asking users to express convex constraints in such a way that convexity is proven by composition rules, which are sufficient but not necessary. These composition rules are those from basic convex analysis, for example, the sum of convex functions is convex, the point-wise maximum of convex functions is convex, and the composition $f(g(x))$ is convex when $f$ is convex and nondecreasing and $g$ is convex. The full set of DCP rules is reviewed in~\cite{DCP,dccp}.

Even though it is possible to write down convex functions which do not satisfy these composition rules, the DCP philosophy is to disallow them and instead introduce new \textit{atoms} (or basic operations) which users must use when writing down their model. For example, $\operatorname{logsumexp}(\left[\begin{array}{cc} x_1 & x_2 \end{array}\right])$ replaces \\ $\log(\exp(x_1)+\exp(x_2))$ and
$\operatorname{norm}(\left[\begin{array}{cc} x_1 & x_2 \end{array}\right])$ replaces $\sqrt{x_1^2 + x_2^2}$.
Although asking users to express their optimization problems in this form breaks away from the traditional setting of MICP, DCP also formalizes the folklore within the MICP community that the way in which you write down the convex constraints can have a significant impact on the solution time; see, e.g., Hijazi et al.~\cite{Hijazi} and our example later discussed in Equation~\ref{eq:tls}.

The success over the past decade of the CVX software package~\cite{cvx} which implements DCP has demonstrated that this modeling concept is practical. Users are willing to learn the rules of DCP in order to gain access to powerful (continuous, convex) solvers, and furthermore the number of basic atoms needed to cover nearly all convex optimization problems of practical interest is relatively small. 

Although we motivated DCP as a solution to the subproblem of convexity detection, it in fact provides a complete solution to the problem of automatically generating an extended formulation and encoding it in a computationally convenient form given a user's algebraic representation of a problem. 
All DCP-valid expressions are compositions of basic operations (atoms); for example the expression $\max\{\exp(x^2),-2x\}$ is DCP-valid because the basic composition rules prove its convexity. A lesser-known aspect of DCP is that these rules of composition have a 1-1 correspondence with extended formulations based on the epigraphs of the atoms. Observe, for example, that 
\begin{equation}\label{eq:cvxepi}
t \ge \max\{\exp(x^2),-2x\}
\end{equation}
if and only if
\begin{equation}
t \ge \exp(x^2), t \ge -2x
\end{equation}
if and only if there exists $s$ such that
\begin{equation}\label{eq:cvxext}
s \ge x^2, t \ge \exp(s), t \ge -2x,
\end{equation}
where the validity of the latter transformation holds precisely because $\exp(\cdot)$ is increasing and therefore $s \ge x^2$ implies $\exp(s) \ge \exp(x^2)$. Furthermore, the constraints $s \ge x^2$ and $t \ge \exp(s)$ are convex because square and $\exp$ are convex functions; hence~\eqref{eq:cvxext} is a convex extended formulation of~\eqref{eq:cvxepi}. Note that while we previously discussed extended formulations derived only from disaggregating sums, disaggregating compositions of functions in this form also yields stronger polyhedral approximations~\cite{Baron}. The existence of this extended formulation is no coincidence. Grant and Boyd~\cite{gb08} explain that a tractable representation of the epigraph of an atom is sufficient to incorporate it into a DCP modeling framework. That is, if an implementation of DCP knows how to optimize over a model with the constraint $t \ge f(x)$ for some convex function $f$, then $f$ can be incorporated as an atom within the DCP framework and used within much more complex expressions so long as they follow the DCP composition rules.

Our analysis of DCP has led us to the 
conclusion that DCP provides the means to automate the generation of extended formulations in a way that has never been done in the context of MICP. Users need only express their MICP problem by using a DCP modeling language like CVX or more recent implementations like CVXPY~\cite{cvxpy} (in Python), or Convex.jl~\cite{convexjl} (in Julia). Any DCP-compatible model is convex by construction and emits an extended formulation that can safely disaggregate sums and more complex compositions of functions.

We do note that in some cases it may not be obvious how to write a known convex function in DCP form. In our work described in~\cite{IPCO} where we translated MICP benchmark instances into DCP form, we were unable to find a DCP representation of the univariate concave function $\frac{x}{x+1}$ which is not in DCP form because division of affine expressions is neither convex nor concave in general. Fortunately, a reviewer suggested rewriting $\frac{x}{x+1} = 1- \frac{1}{x+1}$ where $\frac{1}{x+1}$ is a DCP-recognized convex function so long as $x+1 \ge 0$. With this trick we were able to translate \textit{all} of the benchmark instances we considered into DCP form, as we discuss in more details in the following section.

\section{MIDCP and conic representability}\label{sec:micone}

While DCP modeling languages have traditionally supported only convex problems, CVX recently added support for mixed-integer convex problems under the name of MI\textbf{D}CP, and the subsequently-developed DCP modeling languages also support integer constraints. We will use the terminology MIDCP to refer to MICP models expressed in DCP form. In the previous section we argued that an MIDCP representation of an MICP problem provides sufficient information to construct an extended formulation, which in turn could be used to accelerate the convergence of the outer approximation algorithm by providing strong polyhedral approximations. However, an MIDCP representation is quite complex, much more so than the ``black-box'' derivative-based representation that traditional MICP solvers work with. Handling the MIDCP form requires understanding each atom within the DCP library and manipulating the \textit{expression graph} data structures which are used to represent the user's algebraic expressions.

It turns out that there is a representation of MIDCP models which is much more compact and convenient for use as an input format for an MICP solver, and this is as \textit{mixed-integer conic} optimization problems. Before stating the form of these problems, we first consider the standard continuous conic optimization problem:
\begin{align}\label{eq:conic}
\min_{x} \quad& c^Tx\notag\\
\text{s.t.}\quad & Ax = b\tag{CONE}\\
& x \in \mathcal{K},\notag
\end{align}
where $\mathcal{K}\subseteq \mathbb{R}^n$ is a closed convex cone, that is, a closed convex set $\mathcal{K}$ where any nonnegative scaling $\alpha x$ of a point $x$ in the set remains in the set. A simple example of a cone is the nonnegative orthant $\mathbb{R}_+^n = \{  x \in \mathbb{R}^n : x \ge 0 \}$. When $\mathcal{K} = \mathbb{R}_+^n$ then~\eqref{eq:conic} reduces to a linear programming problem. 
Typically, $\mathcal{K}$ is a product of cones $\mathcal{K}_1 \times \mathcal{K}_2
 \times \cdots \times \mathcal{K}_r$, where each $\mathcal{K}_i$ is one of a small number of recognized cones.
 
One of Grant et al.'s original motivations for developing the DCP framework was to provide access to powerful solvers for the second-order cone~\cite{SOCPApplications}, 
\begin{align}
\text{SOC}_n = \{ (t,x) \in \mathbb{R}^n : ||x|| \le t \},
\end{align}
and the cone of positive semidefinite matrices,
\begin{align}
    \text{PSD}_n = \{ A \in \mathbb{R}^{n\times n} : A = A^T, x^TAx \ge 0\, \forall\, x \in \mathbb{R}^n \}.
\end{align}
CVX, for example, does \textit{not} use smooth, derivative-based representations of the epigraphs of atoms but instead uses a conic representation of each of its atoms. For instance, for $x,y \ge 0$ the epigraph of the negated geometric mean $f(x,y) = -\sqrt{xy}$ is a convex set representable as $t \ge -\sqrt{xy}$ iff $\exists\, z \ge 0$ such that $-t \le z \le \sqrt{xy}$ iff
\begin{equation}
-t \le z \text{ and } z^2 \le xy \text{ iff } -t \le z \text{ and } (x/\sqrt{2},y/\sqrt{2},z) \in \text{RSOC}_3,
\end{equation}
where
\begin{equation}\label{eq:rsoc}
\text{RSOC}_n := \{ (x,y,z) \in \mathbb{R}\times\mathbb{R}\times\mathbb{R}^{n-2} : 2xy \ge ||z||_2^2, x \ge 0, y \ge 0 \}
\end{equation}
is the $n$-dimensional \textit{rotated} second-order cone, a common cone useful for modeling (e.g., also for functions like $x^2$) which itself is representable as a transformation of the second-order cone~\cite{lectures}. While this conic representation of the geometric mean is known in the literature~\cite{lectures}, it is arguably unnecessarily complex for modelers to understand, and CVX, for example, provides a \texttt{geo\_mean} atom which transparently handles this transformation. 

Subsequent to the second-order and positive semidefinite cones, researchers have investigated the exponential cone~\cite{akle},
\begin{align}
    \text{EXP} = \operatorname{cl}\{ (x,y,z) \in \mathbb{R}^3: y\exp(x/y) \le z, y > 0 \},
\end{align}
and the power cone~\cite{powercone},
\begin{align}
\text{POW}_\alpha = \{ (x,y,z) \in \mathbb{R}^3 : |z| \le x^\alpha y^{1-\alpha}, x \ge 0, y \ge 0\},
\end{align}
which can be used to represent functions like entropy ($-x\log(x)$) and fractional powers, respectively. This small collection of cones is sufficient to represent any convex optimization problem which you may input within existing DCP implementations, including CVX.

In the context of MICP, these cones are indeed sufficient from our experience. We classified all 333 MICP instances from the MINLPLIB2 benchmark library~\cite{MINLPLIB} and found that 217 are representable by using purely second-order cones (and so fall under the previously mentioned MISOCP problem class), 107 are representable by using purely exponential cones, and the remaining by some mix of second-order, exponential, and power cones. We refer readers to~\cite{IPCO} for an extended discussion of conic representability. Of particular note are the trimloss~\cite{Harjunkoski} family of instances which have constraints of the form,
 \begin{equation}\label{eq:tls}
\sum\nolimits_{k=1}^q -\sqrt{x_k y_k} \le c^Tz + b.
\end{equation}
Prior to our report in~\cite{IPCO}, the \texttt{tls5} and \texttt{tls6} instances had been unsolved since 2001. By directly rewriting these problems into MIDCP form, we obtained an MISOCP representation because all constraints are representable by using second-order cones, precisely by using the transformation of the geometric mean discussed above. Once in MISOCP form, we provided the problem to Gurobi 6.0, which was able to solve them to global optimality within a day, indicating the value of conic formulations.

Given that DCP provides an infrastructure to translate convex problems into conic form, we may consider mixed-integer conic problems as a compact representation of MIDCP problems. Below, we state our standard form for mixed-integer conic problems,
\begin{align}\label{eq:miconic}
\min_{x,z} \quad& c^Tz\notag\\
\text{s.t.}\quad & A_xx + A_zz = b\tag{MICONE}\\
& L \le x \le U, x \in \mathbb{Z}^n, z \in \mathcal{K},\notag
\end{align}
where $\mathcal{K}\subseteq \mathbb{R}^k$ is a closed convex cone. Without loss of generality, we assume integer variables are not restricted to cones, since we may introduce corresponding continuous variables by equality constraints. In Section~\ref{sec:conicoa} we discuss solving~\eqref{eq:miconic} via polyhedral outer approximation.

\section{Outer approximation algorithm for mixed-integer conic problems}\label{sec:conicoa}

The observations of the previous section motivated the development of
a solver for problems of the form~\eqref{eq:miconic}. In~\cite{IPCO} we developed the first outer-approximation algorithm with finite-time convergence guarantees for such problems. We note that the traditional convergence theory is generally insufficient because it assumes differentiability, while conic problems have nondifferentiability that is sometimes intrinsic to the model. Nonsmooth perspective functions like $f(x,y) = x^2/y$, for example, which are used in disjunctive convex optimization~\cite{CeriaSoares}, have been particularly challenging for derivative-based MICP solvers and have motivated smooth approximations~\cite{perspective}. On the other hand, conic form can handle these nonsmooth functions in a natural way, so long as there is a solver capable of solving the continuous conic relaxations.

In this section, we provide an overview of the key points of the algorithm and refer readers to~\cite{IPCO} for the full description. The finite-time convergence guarantees of the outer approximation algorithm depend on an assumption that strong duality holds in certain convex subproblems. Extending~\cite{IPCO}, we include a discussion on what may happen when this assumption does not hold.

We begin with the definition of dual cones.

\begin{definition}
Given a cone $\mathcal{K}$, we define $\mathcal{K}^* := \{ \beta \in \mathbb{R}^k : \beta^Tz \ge 0 \,\, \forall z \in \mathcal{K}\}$ as the dual cone of $\mathcal{K}$.
\end{definition}

Dual cones provide an equivalent outer description of any closed, convex cone, as the following lemma states. We refer readers to~\cite{lectures} for the proof.

\begin{lemma}
Let $\mathcal{K}$ be a closed, convex cone. Then $z \in \mathcal{K}$ iff $z^T\beta \ge 0\,\, \forall \beta \in \mathcal{K}^*$.
\end{lemma}
We note that the second-order cone $\text{SOC}_n$, the rotated second order cone $\text{RSOC}_n$~\eqref{eq:rsoc}, and the cone of positive semidefinite matrices are \textit{self-dual}, which means that the dual cone and the original cone are the same~\cite{lectures}. While the exponential and power cones are not self-dual, the discussions that follow are valid for them and other general cones.

Based on the above lemma, we state the analogue of the MILP relaxation\\~\ref{eq:mioa1} for~\eqref{eq:miconic} as
\begin{align}\label{eq:mioa}
\min_{x,z} \quad& c^Tz\notag\\
\text{s.t.}\quad & A_xx + A_zz = b\tag{MICONEOA(T)}\\
& L \le x \le U, x \in \mathbb{Z}^n,\notag\\
&\beta^Tz \ge 0\,\, \forall \beta \in T.\notag
\end{align}

Note that if $T = \mathcal{K}^*$, \ref{eq:mioa} is an equivalent semi-infinite representation of~\eqref{eq:miconic}. If $T \subset \mathcal{K}^*$ and $|T| < \infty$ then~\ref{eq:mioa} is an MILP outer approximation of~\eqref{eq:miconic} whose objective value is a lower bound on the optimal value of~\eqref{eq:miconic}. In the context of the discussion in Section~\ref{sec:OA}, given $T$, our polyhedral approximation of $\mathcal{K}$ is $P_T = \{ z : \beta^Tz \ge 0\,\,\forall\, \beta \in T\}$, and we explicitly treat the linear equality constraints separately.

In the conic setting, we state the continuous subproblem~\ref{eq:convsub} with integer values fixed as
\begin{align}\label{eq:conic_cont}
v_{x^*} = \min_{z} \quad&  c^Tz\notag\\
\text{s.t.}\quad & A_zz = b - A_x x^*\tag{CONE($x^*$)},\\
&z \in \mathcal{K}\notag.
\end{align}
Using conic duality, we obtain the dual of~\ref{eq:conic_cont} as
\begin{align}\label{eq:conic_cont_dual}
\max_{\beta,\lambda} \quad& \lambda^T(b-A_x x^*)\notag\\
\text{s.t.}\quad & \beta = c - A_z^T\lambda\\
&\beta \in \mathcal{K}^*\notag.
\end{align}

In~\cite{IPCO} we prove that under the assumptions of strong duality, the optimal solutions $\beta$ to the dual problem~\eqref{eq:conic_cont_dual} correspond precisely to the half-spaces which ensure the conditions in Lemma~\ref{lem:oafinite} when \ref{eq:conic_cont} is feasible; hence, we add these solutions to the set $T$.  When~\ref{eq:conic_cont} is infeasible and~\eqref{eq:conic_cont_dual} is unbounded, the rays of~\eqref{eq:conic_cont_dual} provide solutions that satisfy~\eqref{eq:infeascase}, guaranteeing finite convergence of the OA algorithm.

We previously deferred a discussion of what may happen when the assumption of strong duality fails. We now present a negative result for this case. When the assumption of strong duality fails, it may be impossible for the OA algorithm to converge in a finite number of iterations.

Consider the problem adapted from~\cite{hijaziletter},
\begin{equation}
\begin{split}
    \min \quad& z \\
    \text{s.t.}\quad& x = 0, \\
    & (x,y,z) \in \text{RSOC}_3.
    \end{split}
    \label{eq:coneprob}
\end{equation}
Note that $(0,y,z) \in \text{RSOC}_3$ implies $z = 0$, so the optimal value is trivially zero.

The conic dual of this problem is
\begin{equation}
\begin{split}
    \max\quad& 0\\
    \text{s.t.}\quad&(\beta,0,1) \in \text{RSOC}_3,\\
    & \beta \text{ free}.
\end{split}
\label{eq:coneprobdual}
\end{equation}

The dual is infeasible because there is no $\beta$ satisfying $0\beta \ge 1$. So there is no strong duality in this case. The following lemma demonstrates that polyhedral approximations fail \textit{entirely}. The proof is more technical than the rest of the paper but uses only basic results from linear programming and conic duality.

\begin{lemma}\label{claim:counterexample}
There is no polyhedral outer approximation $P_{\text{RSOC}_3} \supset \text{RSOC}_3$ such that the following relaxation of~\eqref{eq:coneprob} is bounded:
\begin{equation}
\begin{split}
    \min\quad& z \\
    \text{\normalfont s.t.}\quad& x = 0, \\
    & (x,y,z) \in P_{\text{RSOC}_3}.
    \end{split}
    \label{eq:OAprob}
\end{equation}
\end{lemma}
\begin{proof}
Let us assume that $\text{RSOC}_3 \subset P_{\text{RSOC}_3} := \{ (x,y,z) : A_x x + A_y y + A_z z \ge 0 \}$ for some vectors $A_x, A_y$, $A_z$. The right-hand side can be taken to be zero because $\text{RSOC}_3$ is a cone. Specifically, positive right-hand-side values are invalid because they would cut off the point $(0,0,0)$, and negative values can be strengthened to zero.  Since~\eqref{eq:OAprob} is a linear programming problem invariant to positive rescaling, it is bounded iff there exists a feasible dual solution $(\beta,\alpha)$ satisfying
\begin{align}
    \alpha^TA_x = \beta,\\
    \alpha^TA_y = 0,\\
    \alpha^TA_z = 1,\\
    \alpha \ge 0.
\end{align}
Suppose, for contradiction, that there exist $(\beta,\alpha)$ satisfying these dual feasibility conditions.
Let $(A_{x,i},A_{y,i},A_{z,i})$ denote the coefficients of the $i$th linear inequality in $P_{\text{RSOC}_3}$. Since $P_{\text{RSOC}_3}$ is a valid outer approximation, we have that
\begin{equation}
    (A_{x,i},A_{y,i},A_{z,i})\cdot(x,y,z) \ge 0,\, \forall (x,y,z) \in \text{RSOC}_3,
\end{equation}
hence $(A_{x,i},A_{y,i},A_{z,i}) \in (\text{RSOC}_3)^* = \text{RSOC}_3$, recalling that $\text{RSOC}_3$ is self-dual. Therefore we have \begin{equation}
(\alpha^TA_x, \alpha^TA_y, \alpha^TA_z) \in \text{RSOC}_3
\end{equation}
for $\alpha \geq 0$. This follows from the fact that the vector, $(\alpha^TA_x, \alpha^TA_y, \alpha^TA_z)$, is a non-negative linear combination of elements of $\text{RSOC}_3$ and $\text{RSOC}_3$ is a convex cone. However, the duality conditions imply $(\beta, 0, 1) \in \text{RSOC}_3$, i.e., $0 \ge 1$, which is a contradiction.
\end{proof}

Lemma~\ref{claim:counterexample} implies that the following MISOCP instance cannot be solved by the OA algorithm:
\begin{equation}
\begin{split}
    \min\quad& z \\
    \text{s.t.}\quad& x = 0, \\
    & (x,y,z) \in \text{RSOC}_3,\\
    & x \in \{0,1\},
    \end{split}
    \label{eq:miconeprob}
\end{equation}
because the optimal value of \textit{any} MILP relaxation will be $-\infty$ while the true optimal objective is $0$, hence the convergence conditions cannot be satisfied.

This example strengthens the observation by~\cite{hijaziletter} that MISOCP solvers may fail when certain constraint qualifications do not hold. In fact, no approach based on straightforward polyhedral approximation can succeed. Very recently, Gally et al.~\cite{gallymisdp} have studied conditions in the context of mixed-integer semidefinite optimization which ensure that strong duality holds when integer values are fixed.

\section{Computational experiments}\label{sec:computation}

In this section we extend the numerical experiments performed in our previous work~\cite{IPCO}. In that work, we introduced Pajarito. Pajarito is an open-source stand-alone Julia solver, now publicly released\footnote{The results reported here are based on Pajarito version 0.1. The latest release, version 0.4, has been almost completely rewritten with significant algorithmic advances, which will be discussed in upcoming work with Chris Coey.} at \url{https://github.com/JuliaOpt/Pajarito.jl}, that heavily relies on the infrastructure of\\ JuMP~\cite{DunningHuchetteLubin2015}. 

Since~\cite{IPCO} we have improved the performance of Pajarito and report revised numerical experiments. We translated 194 convex instances of MINLPLIB2~\cite{MINLPLIB} into Convex.jl~\cite{convexjl}, a DCP algebraic modeling language in Julia which performs automatic transformation into conic form. Our main points of comparison are Bonmin~\cite{Bonmin} using its OA algorithm, SCIP~\cite{scip} using its default LP-based branch-and-cut algorithm, and Artelys Knitro~\cite{knitro} using its default nonlinear branch-and-bound algorithm; all three can be considered state-of-the-art academic or commercial solvers. We further compare our results with CPLEX for MISOCP instances only. All computations were performed on a high-performance cluster at Los Alamos National Laboratory with Intel$^\circledR$ Xeon$^\circledR$ E5-2687W v3 @3.10GHz 25.6MB L3 cache processors and 251GB DDR3 memory installed on every node. CPLEX v12.6.2 is used as a MILP and MISOCP solver. Because conic solvers supporting exponential cones were not sufficiently reliable in our initial experiments, we use Artelys Knitro v9.1.0 to solve all conic subproblems via traditional derivative-based methods.

Bonmin v1.8.3 and SCIP v3.2.0 are both compiled with CPLEX v12.6.2 and Ipopt v3.12.3 using the HSL linear algebra library MA97. All solvers are set to a relative optimality gap of $10^{-5}$, are run on a single thread (both CPLEX and Artelys Knitro), and are given 10 hours of wall time limit (with the exception of \texttt{gams01}, a previously unsolved benchmark instance, where we give 32 threads to CPLEX for the MILP relaxations). The scripts to run these experiments can be found online at \url{https://github.com/mlubin/MICPExperiments}.

\begin{figure}[t]
\centering
\subfigure[Solution time]{
\includegraphics[scale=0.28]{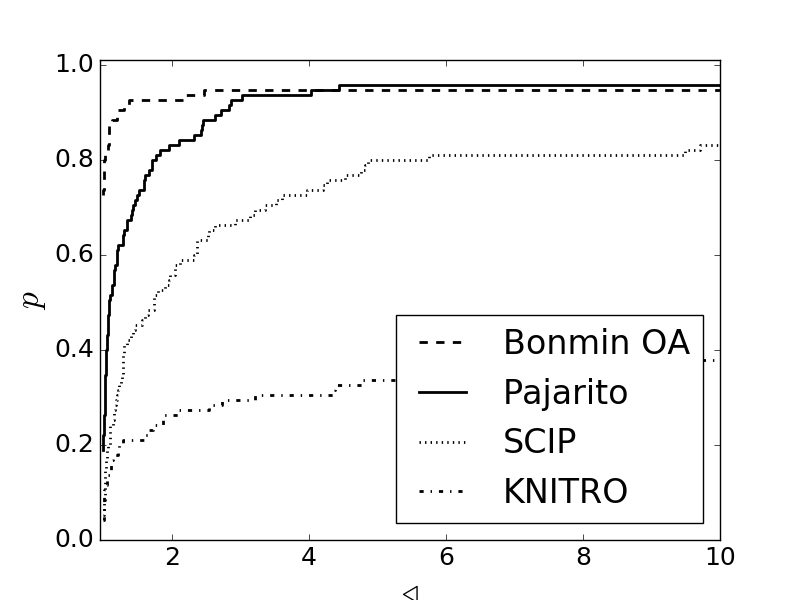}}
\subfigure[Number of OA iterations]{
\includegraphics[scale=0.28]{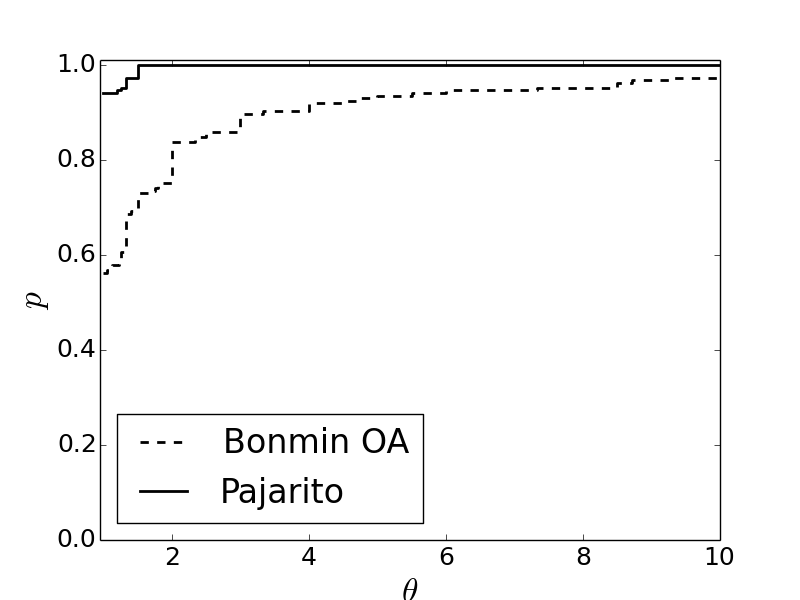}}
\caption{Comparison performance profiles~\cite{perf} (solver performs within a factor of $\theta$ of the best on proportion $p$ of instances) over all instances we tested from the MINLPLIB2 benchmark library. Higher is better. Bonmin is faster than Pajarito often within a small factor, yet Pajarito is able to solve a few more instances overall and with significantly fewer iterations. }
\label{fig:all}
\end{figure}

Numerical experiments indicate that the extended formulation drastically reduces the number of polyhedral OA iterations as expected. In aggregate across the instances we tested, Bonmin requires 2685 iterations while Pajarito requires 994. We list the full results in Tables~\ref{tab:results:1} and \ref{tab:results:2} and summarize them in Figure~\ref{fig:all}. In Figure~\ref{fig:soc} we present results for the subset of SOC-representable instances, where we can compare with commercial MISOCP solvers. In our performance profiles, all times are shifted by 10 seconds to decrease the influence of easy instances.

Notably, Pajarito is able solve a previously unsolved instance, \texttt{gams01}, whose conic representation requires a mix of SOC and EXP cones and hence was not a pure MISOCP problem. The best known bound was 1735.06 and the best known solution was 21516.83. Pajarito solved the instance to optimality with an objective value of 21380.20 in 6 iterations.

\begin{figure}[t]
\centering
\includegraphics[scale=0.4]{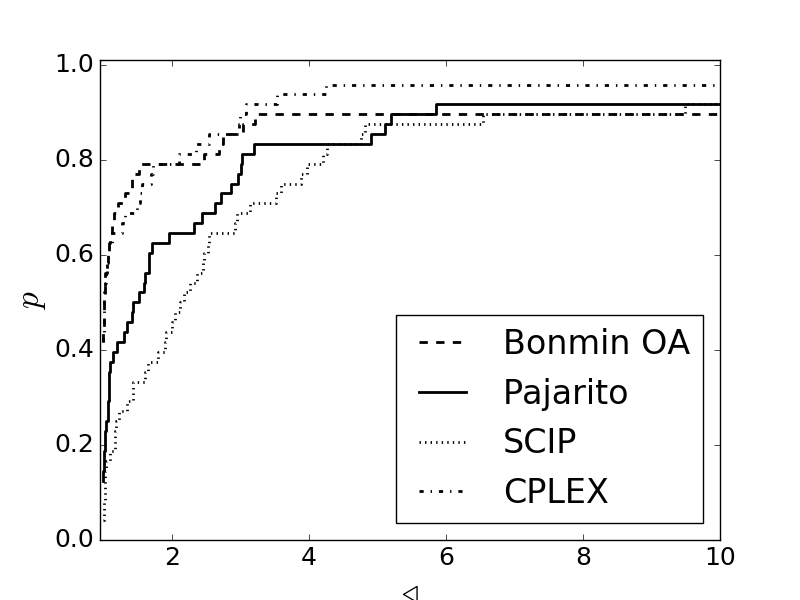}
\caption{Performance profile~\cite{perf} (solver is the fastest within a factor of $\theta$ of the best on proportion $p$ of instances) over the instances representable as mixed-integer second-order cone problems where we can compare with the commercial CPLEX solver. Higher is better. CPLEX is the best overall, since notably it already implements the extended formulation proposed by Vielma et al.~\cite{VielmaExtendedFormulations}.}
\label{fig:soc}
\end{figure}

 \begin{table}[t]
 \tiny
 \centering
 \begin{tabular}{lcrrrrrrr}
 & & \multicolumn{2}{c}{Pajarito} & \multicolumn{2}{c}{Bonmin} & SCIP & Knitro & CPLEX \\
   Instance & Conic rep. & Iter & Time & Iter & Time & Time & Time & Time \\
  \hline
batch & Exp & 1 &     0.26 & 2 &     0.60 &     0.66 &     0.56 & -- \\ 
batchdes & Exp & 1 &     0.11 & 1 &     0.07 &     0.16 &     0.02 & -- \\ 
batchs101006m & Exp & 3 &     3.26 & 10 &     1.88 &     5.10 &    76.96 & -- \\ 
batchs121208m & Exp & 3 &     6.74 & 4 &     3.14 &    13.09 &   316.14 & -- \\ 
batchs151208m & Exp & 3 &    10.72 & 6 &     7.97 &    16.90 &   516.04 & -- \\ 
batchs201210m & Exp & 2 &    25.14 & 8 &    14.92 &    29.12 &   970.51 & -- \\ 
clay0203h & SOC & 5 &     1.42 & 9 &     0.90 &     0.70 &     1.28 &     0.35 \\ 
clay0203m & SOC & 6 &     1.61 & 10 &     0.40 &     0.86 &     0.34 &     0.37 \\ 
clay0204h & SOC & 1 &     1.85 & 3 &     3.60 &     7.14 &     5.72 &     1.61 \\ 
clay0204m & SOC & 1 &     0.55 & 3 &     0.33 &     2.55 &     3.30 &     1.02 \\ 
clay0205h & SOC & 3 &    24.40 & 4 &    20.89 &    78.19 &   168.28 &     8.93 \\ 
clay0205m & SOC & 3 &     8.11 & 6 &     5.50 &     9.63 &    61.91 &     1.77 \\ 
clay0303h & SOC & 5 &     2.41 & 9 &     0.97 &     1.53 &     1.96 &     0.54 \\ 
clay0303m & SOC & 7 &     2.60 & 10 &     0.58 &     1.73 &     0.76 &     0.68 \\ 
clay0304h & SOC & 9 &    13.87 & 11 &     5.27 &     2.50 &    26.33 &     1.42 \\ 
clay0304m & SOC & 13 &    18.97 & 16 &     2.84 &     7.09 &     7.20 &     2.13 \\ 
clay0305h & SOC & 3 &    52.97 & 4 &    23.81 &     1.97 &   139.27 &    23.32 \\ 
clay0305m & SOC & 3 &    11.83 & 7 &     6.16 &    12.90 &    52.53 &     2.51 \\ 
du-opt & SOC & 7 &     3.19 & 61 &     0.76 & \textgreater 36000 &     0.11 &     1.54 \\ 
du-opt5 & SOC & 4 &     1.55 & 22 &     0.22 &     0.75 &     0.11 &     1.97 \\ 
enpro48pb & Exp & 1 &     0.51 & 2 &     0.22 &     1.73 &     0.85 & -- \\ 
enpro56pb & Exp & 1 &     0.60 & 1 &     0.22 &     1.52 &     4.47 & -- \\ 
ex1223 & ExpSOC & 1 &     0.06 & 3 &     0.07 &     0.14 &     0.03 & -- \\ 
ex1223a & SOC & 0 &     0.02 & 1 &     0.03 &     0.11 &     0.02 &     0.01 \\ 
ex1223b & ExpSOC & 1 &     0.08 & 3 &     0.07 &     0.15 &     0.02 & -- \\ 
ex4 & SOC & 2 &     1.06 & 2 &     0.13 &     1.15 &     0.25 &     0.86 \\ 
fac3 & SOC & 2 &     0.19 & 6 &     0.15 &     0.24 &     0.16 &     0.07 \\ 
netmod\_dol2 & SOC & 7 &    49.97 & 33 &   167.49 &    33.93 &   293.76 &    12.58 \\ 
netmod\_kar1 & SOC & 12 &     8.05 & 102 &    56.45 &     3.32 &   122.98 &     7.68 \\ 
netmod\_kar2 & SOC & 12 &     8.14 & 102 &    56.35 &     3.30 &   122.28 &     7.66 \\ 
no7\_ar25\_1 & SOC & 3 &    67.97 & 2 &    25.19 &    82.09 & 17601.54 &    54.34 \\ 
no7\_ar2\_1 & SOC & 1 &     8.87 & 1 &     7.06 &    31.81 & 14957.66 &    21.83 \\ 
no7\_ar3\_1 & SOC & 3 &    91.36 & 4 &    71.04 &   392.98 & 16495.95 &   126.09 \\ 
no7\_ar4\_1 & SOC & 4 &   107.58 & 5 &    85.87 &   274.72 & 17865.83 &    48.97 \\ 
no7\_ar5\_1 & SOC & 5 &   115.25 & 7 &    69.23 &    68.90 & 17452.47 &    32.60 \\ 
nvs03 & SOC & 1 &     0.03 & 1 &     0.06 &     0.13 &     0.18 &     0.00 \\ 
slay04h & SOC & 2 &     0.32 & 5 &     0.19 &     0.68 &     0.53 &     0.14 \\ 
slay04m & SOC & 2 &     0.17 & 5 &     0.11 &     0.57 &     0.32 &     0.18 \\ 
slay05h & SOC & 3 &     0.65 & 9 &     0.60 &     3.29 &     1.57 &     0.37 \\ 
slay05m & SOC & 3 &     0.28 & 7 &     0.18 &     0.84 &     1.02 &     0.16 \\ 
slay06h & SOC & 2 &     0.76 & 12 &     1.94 &     5.26 &     4.65 &     0.69 \\ 
slay06m & SOC & 2 &     0.32 & 9 &     0.29 &     1.57 &     2.94 &     0.42 \\ 
slay07h & SOC & 3 &     1.75 & 15 &     5.04 &    18.35 &     9.96 &     0.98 \\ 
slay07m & SOC & 3 &     0.56 & 12 &     0.66 &     2.51 &     5.75 &     0.67 \\ 
slay08h & SOC & 3 &     2.65 & 22 &    27.27 &   180.20 &    26.47 &     1.50 \\ 
slay08m & SOC & 2 &     0.58 & 21 &     2.89 &     3.69 &    13.17 &     0.96 \\ 
slay09h & SOC & 3 &     4.36 & 36 &   163.31 &    92.70 &    79.79 &     1.93 \\ 
slay09m & SOC & 3 &     1.11 & 28 &    17.22 &    11.01 &    33.36 &     1.54 \\ 
slay10h & SOC & 4 &    21.94 & 80 &  8155.02 & 11745.37 &   442.46 &     7.55 \\ 
slay10m & SOC & 4 &     4.36 & 77 &  1410.08 &   516.81 &   167.81 &     1.80 \\ 
syn05h & Exp & 1 &     0.07 & 2 &     0.09 &     0.31 &     0.17 & -- \\ 
syn05m & Exp & 1 &     0.04 & 2 &     0.07 &     0.28 &     0.14 & -- \\ 
syn05m02h & Exp & 1 &     0.15 & 1 &     0.06 &     0.33 &     0.11 & -- \\ 
syn05m02m & Exp & 1 &     0.08 & 1 &     0.07 &     0.33 &     0.29 & -- \\ 
syn05m03h & Exp & 1 &     0.23 & 1 &     0.07 &     0.33 &     0.13 & -- \\ 
syn05m03m & Exp & 1 &     0.12 & 1 &     0.07 &     0.32 &     0.30 & -- \\ 
syn05m04h & Exp & 1 &     0.29 & 1 &     0.07 &     0.38 &     0.19 & -- \\ 
syn05m04m & Exp & 1 &     0.17 & 1 &     0.08 &     0.32 &     0.61 & -- \\ 
syn10h & Exp & 0 &     0.10 & 1 &     0.04 &     0.20 &     0.09 & -- \\ 
syn10m & Exp & 1 &     0.08 & 2 &     0.04 &     0.25 &     0.23 & -- \\ 
syn10m02h & Exp & 1 &     0.27 & 1 &     0.09 &     0.46 &     0.21 & -- \\ 
syn10m02m & Exp & 1 &     0.16 & 2 &     0.09 &     0.42 &     3.05 & -- \\ 
syn10m03h & Exp & 1 &     0.38 & 1 &     0.08 &     0.59 &     0.24 & -- \\ 
syn10m03m & Exp & 1 &     0.23 & 1 &     0.08 &     0.54 &    10.47 & -- \\ 
syn10m04h & Exp & 1 &     0.53 & 1 &     0.11 &     0.52 &     0.19 & -- \\ 
syn10m04m & Exp & 1 &     0.34 & 1 &     0.11 &     0.72 &    40.41 & -- \\ 
syn15h & Exp & 1 &     0.22 & 1 &     0.06 &     0.29 &     0.14 & -- \\ 
syn15m & Exp & 1 &     0.10 & 2 &     0.07 &     0.30 &     0.32 & -- \\ 
syn15m02h & Exp & 1 &     0.51 & 1 &     0.09 &     0.47 &     0.18 & -- \\ 
syn15m02m & Exp & 1 &     0.24 & 1 &     0.09 &     0.44 &     5.51 & -- \\ 
syn15m03h & Exp & 1 &    44.15 & 1 &     0.13 &     0.99 &     0.23 & -- \\ 
syn15m03m & Exp & 1 &     0.38 & 2 &     0.11 &     0.66 &    25.67 & -- \\ 
syn15m04h & Exp & 1 &     1.47 & 1 &     0.14 &     1.61 &     0.32 & -- \\ 
syn15m04m & Exp & 1 &     0.50 & 2 &     0.14 &     1.43 &   186.20 & -- \\ 
syn20h & Exp & 2 &     0.33 & 2 &     0.10 &     0.34 &     0.20 & -- \\ 
syn20m & Exp & 1 &     0.13 & 2 &     0.06 &     0.39 &     1.31 & -- \\ 
syn20m02h & Exp & 2 &     1.07 & 2 &     0.15 &     0.57 &     0.41 & -- \\ 
syn20m02m & Exp & 2 &     0.44 & 2 &     0.10 &     0.73 &   381.88 & -- \\ 
syn20m03h & Exp & 1 &     1.21 & 1 &     0.13 &     1.52 &     0.78 & -- \\ 
syn20m03m & Exp & 2 &     0.64 & 2 &     0.15 &     2.00 &   993.73 & -- \\ 
syn20m04h & Exp & 1 &     1.81 & 1 &     0.19 &     2.41 &     1.11 & -- \\ 
syn20m04m & Exp & 2 &     0.91 & 2 &     0.27 &     9.77 &  1806.83 & -- \\ 
syn30h & Exp & 3 &     0.73 & 3 &     0.12 &     0.59 &     0.28 & -- \\ 
syn30m & Exp & 3 &     0.28 & 3 &     0.09 &     0.49 &    90.26 & -- \\ 
syn30m02h & Exp & 3 &     1.77 & 3 &     0.21 &    12.98 &     0.44 & -- \\ 
syn30m02m & Exp & 3 &     0.82 & 4 &     0.19 &     1.67 &  1041.13 & -- \\ 
syn30m03h & Exp & 3 &     2.24 & 3 &     0.40 & 11444.39 &     1.23 & -- \\ 
syn30m03m & Exp & 3 &     1.28 & 3 &     0.27 &     7.78 &  1878.32 & -- \\ 
syn30m04h & Exp & 3 &     3.51 & 3 &     0.49 & \textgreater 36000 &     2.74 & -- \\ 
syn30m04m & Exp & 3 &     1.81 & 4 &     0.42 &    37.94 &  3113.33 & -- \\ 
syn40h & Exp & 3 &     0.92 & 4 &     0.19 &     0.55 &     0.33 & -- \\ 
syn40m & Exp & 2 &     0.35 & 4 &     0.97 &     0.52 &   484.94 & -- \\ 
syn40m02h & Exp & 3 &     2.15 & 3 &     0.31 &  2073.62 &     1.03 & -- \\ 
syn40m02m & Exp & 3 &     1.18 & 3 &     0.24 &     5.74 &  1550.39 & -- \\ 
syn40m03h & Exp & 4 &     4.20 & 4 &     0.59 &     2.88 &     5.27 & -- \\ 
syn40m03m & Exp & 4 &     2.33 & 5 &     0.52 &   204.94 &  2921.63 & -- \\ 
syn40m04h & Exp & 3 &     8.56 & 4 &     1.02 & \textgreater 36000 &    20.31 & -- \\ 
syn40m04m & Exp & 5 &     4.61 & 5 &     0.87 &   974.05 &  8048.34 & -- \\ 
 \end{tabular}
 \vspace{0.3cm}
 \caption{MINLPLIB2 instances. ``Conic rep'' column indicates which cones are used in the conic representation of the instance (second-order cone and/or exponential). CPLEX is capable of solving only second-order cone instances. Times in seconds.}
 \label{tab:results:1}
 \end{table}

 \begin{table}[t]
 \tiny
 \centering
 \begin{tabular}{lcrrrrrrr}
 & & \multicolumn{2}{c}{Pajarito} & \multicolumn{2}{c}{Bonmin} & SCIP & Knitro & CPLEX \\
   Instance & Conic rep. & Iter & Time & Iter & Time & Time & Time & Time \\
  \hline
synthes1 & Exp & 2 &     0.06 & 3 &     0.04 &     0.24 &     0.11 & -- \\ 
synthes2 & Exp & 2 &     0.07 & 3 &     0.05 &     0.42 &     0.13 & -- \\ 
synthes3 & Exp & 2 &     0.09 & 6 &     0.10 &     0.34 &     0.13 & -- \\ 
rsyn0805h & Exp & 1 &     0.38 & 1 &     0.14 &     0.40 &     1.10 & -- \\ 
rsyn0805m & Exp & 2 &     0.49 & 2 &     0.25 &     0.87 &    53.62 & -- \\ 
rsyn0805m02h & Exp & 5 &     2.38 & 5 &     0.71 &     0.73 &     3.71 & -- \\ 
rsyn0805m02m & Exp & 4 &     2.41 & 4 &     2.16 &    11.21 &  1617.65 & -- \\ 
rsyn0805m03m & Exp & 3 &     3.26 & 3 &     4.08 &    10.71 &  2930.70 & -- \\ 
rsyn0805m04m & Exp & 2 &     2.32 & 2 &     2.31 &    19.17 &  5202.46 & -- \\ 
rsyn0810m & Exp & 1 &     0.37 & 2 &     0.24 &     1.17 &   211.18 & -- \\ 
rsyn0810m02h & Exp & 3 &     1.87 & 3 &     0.58 &     1.61 &     9.79 & -- \\ 
rsyn0810m02m & Exp & 3 &     2.20 & 4 &     5.78 &    49.36 &  3098.62 & -- \\ 
rsyn0810m03h & Exp & 3 &     3.19 & 3 &     1.36 &     1.99 &    26.42 & -- \\ 
rsyn0810m03m & Exp & 3 &     4.29 & 3 &     6.04 &    41.61 &  3582.39 & -- \\ 
rsyn0810m04h & Exp & 2 &     3.54 & 3 &     1.31 &     2.87 &     8.61 & -- \\ 
rsyn0810m04m & Exp & 3 &     3.74 & 4 &     3.77 &    52.17 &  5943.63 & -- \\ 
rsyn0815h & Exp & 1 &    19.15 & 1 &     0.27 &     1.27 &     1.77 & -- \\ 
rsyn0815m & Exp & 2 &     0.49 & 2 &     0.23 &     1.21 &   171.89 & -- \\ 
rsyn0815m02m & Exp & 4 &     2.39 & 5 &     1.94 &    58.70 &  2565.52 & -- \\ 
rsyn0815m03h & Exp & 5 &    11.58 & 5 &     5.21 &    38.80 &    31.62 & -- \\ 
rsyn0815m03m & Exp & 5 &     5.66 & 4 &     4.59 &   217.30 &  3914.97 & -- \\ 
rsyn0815m04h & Exp & 3 &     6.16 & 3 &     2.03 &     4.73 &    20.55 & -- \\ 
rsyn0815m04m & Exp & 3 &     6.40 & 4 &     7.78 &  1609.07 &  7313.05 & -- \\ 
rsyn0820h & Exp & 2 &     1.02 & 3 &     0.42 &     2.04 &     1.55 & -- \\ 
rsyn0820m & Exp & 2 &     0.61 & 2 &     0.24 &     3.74 &   772.36 & -- \\ 
rsyn0820m02h & Exp & 2 &     2.28 & 3 &     0.59 &     2.83 &    90.89 & -- \\ 
rsyn0820m02m & Exp & 3 &     2.27 & 3 &     1.90 &   712.08 &  3138.98 & -- \\ 
rsyn0820m03h & Exp & 2 &     3.55 & 2 &     1.37 &     4.72 &   135.69 & -- \\ 
rsyn0820m03m & Exp & 3 &     4.08 & 3 &     5.14 &  6372.80 &  5220.60 & -- \\ 
rsyn0820m04h & Exp & 4 &     7.75 & 4 &     2.66 &     6.25 &    50.72 & -- \\ 
rsyn0820m04m & Exp & 3 &     7.22 & 3 &     8.65 & 13412.29 &  8314.96 & -- \\ 
rsyn0830h & Exp & 3 &     1.27 & 3 &     0.41 &     2.53 &     2.84 & -- \\ 
rsyn0830m & Exp & 4 &     0.96 & 4 &     0.37 &     3.37 &  1012.27 & -- \\ 
rsyn0830m02m & Exp & 5 &    10.95 & 5 &     1.83 &   131.12 &  9151.72 & -- \\ 
rsyn0830m03h & Exp & 2 &     4.77 & 2 &     1.45 &     6.70 &    59.98 & -- \\ 
rsyn0830m03m & Exp & 4 &     5.79 & 4 &     3.45 &  4044.25 & 10519.40 & -- \\ 
rsyn0830m04h & Exp & 3 &     8.44 & 3 &     2.35 &    14.23 &   209.80 & -- \\ 
rsyn0830m04m & Exp & 4 &    11.62 & 4 &    11.47 & \textgreater 36000 & 12709.29 & -- \\ 
rsyn0840h & Exp & 2 &     1.15 & 2 &     0.30 &     3.22 &     0.94 & -- \\ 
rsyn0840m & Exp & 3 &     0.86 & 2 &     0.26 &     2.96 &  1117.90 & -- \\ 
rsyn0840m02h & Exp & 2 &     2.97 & 3 &     0.72 &     5.10 &     8.43 & -- \\ 
rsyn0840m02m & Exp & 3 &     3.05 & 4 &     1.53 &   675.24 &  4443.70 & -- \\ 
rsyn0840m03h & Exp & 3 &     7.24 & 3 &     1.85 & \textgreater 36000 &    41.84 & -- \\ 
rsyn0840m03m & Exp & 5 &     7.92 & 5 &     2.47 &  4662.04 & 10511.67 & -- \\ 
rsyn0840m04h & Exp & 2 &    40.03 & 2 &     2.40 &    18.71 &   453.32 & -- \\ 
rsyn0840m04m & Exp & 4 &    18.14 & 4 &     7.62 & \textgreater 36000 & 15336.01 & -- \\ 
gbd & SOC & 0 &     0.01 & 1 &     0.04 &     0.19 &     0.12 &     0.00 \\ 
ravempb & Exp & 1 &     0.79 & 4 &     0.33 &     0.80 &     0.42 & -- \\ 
portfol\_classical050\_1 & SOC & 12 &    32.66 & 989 & \textgreater 36000 &   133.43 &   452.49 &     3.31 \\ 
m3 & SOC & 0 &     0.04 & 1 &     0.68 &     0.33 &     0.38 &     0.07 \\ 
m6 & SOC & 1 &     0.39 & 1 &     0.16 &     2.07 &   658.83 &     0.17 \\ 
m7 & SOC & 0 &     0.42 & 1 &     0.59 &     4.99 & 10431.03 &     0.69 \\ 
m7\_ar25\_1 & SOC & 1 &     0.55 & 1 &     0.37 &     1.90 &  2763.66 &     0.16 \\ 
m7\_ar2\_1 & SOC & 1 &     2.47 & 1 &     2.19 &     5.59 & 14002.89 &     1.58 \\ 
m7\_ar3\_1 & SOC & 1 &     2.33 & 1 &     1.88 &     5.53 & 25222.75 &     0.82 \\ 
m7\_ar4\_1 & SOC & 0 &     0.31 & 1 &     0.35 &     2.08 & 20537.24 &     0.84 \\ 
m7\_ar5\_1 & SOC & 0 &     1.30 & 1 &     0.34 &    11.88 & 38924.33 &     0.98 \\ 
fo7 & SOC & 4 &    38.44 & 3 &    27.68 &    89.18 &  3584.70 &    23.67 \\ 
fo7\_2 & SOC & 2 &    12.52 & 2 &    12.52 &    43.35 &  6298.85 &     4.88 \\ 
fo7\_ar25\_1 & SOC & 4 &    22.95 & 4 &     9.87 &    21.94 & 16685.13 &     9.92 \\ 
fo7\_ar2\_1 & SOC & 3 &    15.19 & 2 &     8.68 &    25.56 & 16123.12 &    11.04 \\ 
fo7\_ar3\_1 & SOC & 3 &    27.00 & 3 &    11.61 &    28.79 & 16539.34 &    22.16 \\ 
fo7\_ar4\_1 & SOC & 2 &    11.31 & 2 &     9.61 &    47.19 & 14674.12 &    10.27 \\ 
fo7\_ar5\_1 & SOC & 1 &     4.44 & 1 &     5.66 &    19.63 & 16634.28 &    12.67 \\ 
fo8 & SOC & 3 &    79.22 & 2 &    79.50 &   145.26 &  6383.13 &    52.92 \\ 
fo8\_ar25\_1 & SOC & 4 &   141.68 & 3 &    45.80 &   121.69 & 23823.27 &    63.09 \\ 
fo8\_ar2\_1 & SOC & 4 &   159.12 & 3 &    59.24 &   319.27 & 19979.89 &    60.09 \\ 
fo8\_ar3\_1 & SOC & 1 &    10.34 & 1 &    14.65 &    70.68 & 20336.26 &    37.85 \\ 
fo8\_ar4\_1 & SOC & 1 &    12.03 & 1 &    10.53 &    62.21 & 21961.80 &    62.60 \\ 
fo8\_ar5\_1 & SOC & 1 &    29.66 & 2 &    23.26 &    94.63 & 24442.99 &    59.75 \\ 
fo9 & SOC & 4 &   210.11 & 3 &   534.56 &  2079.40 &  4200.36 &   227.52 \\ 
fo9\_ar25\_1 & SOC & 6 &  6390.32 & 6 &  1430.17 &  2819.53 & 25608.54 &  1240.89 \\ 
fo9\_ar2\_1 & SOC & 2 &   490.08 & 2 &   205.19 &   896.42 & 19595.03 &   631.46 \\ 
fo9\_ar3\_1 & SOC & 1 &    18.55 & 1 &    16.77 &   730.51 & 24190.96 &   103.84 \\ 
fo9\_ar4\_1 & SOC & 1 &    56.32 & 2 &    40.77 &  1440.47 & 32284.58 &   785.75 \\ 
fo9\_ar5\_1 & SOC & 3 &   131.24 & 2 &    39.47 &   724.35 & 30368.10 &   725.60 \\ 
flay02h & SOC & 2 &     0.10 & 2 &     0.09 &     0.26 &     1.37 &     0.02 \\ 
flay02m & SOC & 2 &     0.06 & 2 &     0.05 &     0.15 &     0.10 &     0.04 \\ 
flay03h & SOC & 8 &     0.98 & 8 &     0.40 &     0.62 &     0.30 &     0.20 \\ 
flay03m & SOC & 8 &     0.44 & 8 &     0.17 &     0.26 &     0.14 &     0.24 \\ 
flay04h & SOC & 24 &    23.43 & 24 &    19.92 &     3.75 &     6.60 &     1.14 \\ 
flay04m & SOC & 22 &     8.24 & 22 &     4.43 &     1.98 &     2.54 &     1.00 \\  
flay05h & SOC & 164 &  6709.06 & 181 &  6583.08 &   221.67 &   357.72 &    96.62 \\ 
flay05m & SOC & 171 &  5030.20 & 180 &  3258.45 &    51.94 &   118.96 &    68.91 \\ 
flay06h & SOC & 31 & \textgreater 36000 & 30 & \textgreater 36000 & 13327.17 &   883.97 &  6958.36 \\ 
flay06m & SOC & 56 & \textgreater 36000 & 68 & \textgreater 36000 &  2803.53 &   279.87 &  4752.04 \\ 
o7 & SOC & 8 &  2778.14 & 9 &  1623.33 &  2074.22 &  3060.64 &   526.94 \\ 
o7\_2 & SOC & 5 &   803.25 & 5 &   435.47 &   899.41 &  6423.68 &   128.95 \\ 
o7\_ar25\_1 & SOC & 3 &   421.01 & 4 &   259.10 &   433.72 & 16789.95 &   455.29 \\ 
o7\_ar2\_1 & SOC & 1 &    72.03 & 1 &    41.51 &   209.30 & 15504.16 &    68.66 \\ 
o7\_ar3\_1 & SOC & 3 &  1041.48 & 4 &   338.68 &   874.36 & 17193.08 &   875.63 \\ 
o7\_ar4\_1 & SOC & 7 &  2665.40 & 7 &  1486.87 &  1080.95 & 17803.19 &   535.17 \\ 
o7\_ar5\_1 & SOC & 4 &   662.44 & 4 &   309.86 &   545.20 & 21972.83 &   216.84 \\ 
o8\_ar4\_1 & SOC & 3 &  7192.54 & 4 &  2736.05 &  6939.85 & 26448.75 &  8447.35 \\ 
o9\_ar4\_1 & SOC & 6 & 14143.93 & 5 &  7248.84 & 34990.47 & 31569.13 & 21722.78 \\ 
{\bf gams01} & {\bf ExpSOC} & {\bf 6} & {\bf 23414.37} & {\bf \textgreater 19} & {\bf \textgreater 36000} & {\bf \textgreater 36000} & {\bf \textgreater 36000} & -- \\
 \end{tabular}
 \vspace{0.3cm}
 \caption{MINLPLIB2 instances, continued.}
 \label{tab:results:2}
 \end{table}

\section{Concluding remarks and future work}

In this work, we have presented and advanced the state-of-the-art in polyhedral approximation techniques for mixed-integer convex optimization problems, in particular exploiting the idea of extended formulations and how to generate them automatically by using disciplined convex programming (DCP). We explain why the mixed-integer conic view of mixed-integer convex optimization is surprisingly powerful, precisely because it encodes the extended formulation structure in a compact way. We claim that for the vast majority of problems in practice, conic forms using a small number of recognized cones is a sufficient and superior representation to the traditional smooth ``black box'' view.

Our developments for mixed-integer conic optimization seem to have outpaced the capabilities of existing conic solvers, and we hope that the convex optimization community will continue to develop techniques and publicly available, numerically robust solvers in particular for nonsymmetric cones like the exponential cone. In spite of some numerical troubles when solving the conic subproblems using existing solvers, our new mixed-integer conic solver, Pajarito, has displayed superior performance in many cases to state-of-the-art solvers like Bonmin, including the solution of previously unsolved benchmark problems.

This work has opened up a number of promising directions which we are currently pursuing. In the near term we plan on composing a rigorous report on the technical aspects of implementing the outer-approximation algorithm for mixed-integer conic problems, including aspects we have omitted which are important for the reliability and stability of Pajarito. These will include a larger set of benchmark instances and experiments with a branch-and-cut variant of the algorithm.

We intend to investigate the application of polyhedral approximation in the context of mixed-integer semidefinite optimization, where we expect that failures in strong duality could be a common occurrence based on the reports of~\cite{gallymisdp}. It remains an open question what guidance we can provide to modelers on how to avoid cases where polyhedral approximation can fail, or even if this could be resolved automatically at the level of DCP.

Finally, we note that neither the DCP representation of a problem nor the conic representation of a DCP atom is necessarily unique. Understanding the effects of different formulation choices is an important avenue for future work.

\begin{acknowledgements}
We thank Chris Coey for proofreading and Madeleine Udell and the anonymous reviewers for their comments.
M. Lubin was supported by the DOE Computational Science Graduate
Fellowship, which is provided under grant number DE-FG02-97ER25308.
The work at LANL was funded by the Center for Nonlinear Studies (CNLS) and was carried out under the auspices
of the NNSA of the U.S.
DOE at LANL
under Contract No. DE-AC52-06NA25396. J.P. Vielma was funded by NSF grant CMMI-1351619.
\end{acknowledgements}

\bibliographystyle{spmpsci}      
\bibliography{refs}   

\end{document}